\documentclass[11pt,a4paper]{amsart}
\usepackage[T1]{fontenc}
\usepackage[utf8]{inputenc}
\usepackage{lmodern}
\usepackage{amsmath,amssymb,amsthm,mathtools}
\usepackage[margin=29mm]{geometry}
\usepackage{microtype}
\usepackage{color}
\usepackage[authoryear,round]{natbib}
\setcitestyle{aysep={},notesep={, }}
\usepackage[hidelinks]{hyperref}
\newtheorem{theorem}{Theorem}[section]

\newtheorem{lemma}[theorem]{Lemma}
\newtheorem{corollary}[theorem]{Corollary}
\theoremstyle{definition}

\theoremstyle{remark}

\newcommand{\Pcal}{\mathcal P}
\newcommand{\Ucal}{\mathcal U}

\newcommand{\Vcal}{\mathcal V}
\newcommand{\Ncal}{\mathcal N}
\newcommand{\R}{\mathbb R}

\DeclareMathOperator*{\argmin}{arg\,min}

\numberwithin{equation}{section}

\newcommand\commentout[1]{}

\hypersetup{pdfauthor={Stephan Huckemann and Alexander Lytchak},
 pdfsubject={Frechet means and a topological obstruction to stable uniqueness},
 pdfkeywords={Frechet mean, Wasserstein space, reach, Riemannian manifold, Alexandrov space, nonuniqueness}}
\theoremstyle{plain}
\newtheorem{maintheorem}{Theorem}

\title[Dense uniqueness and dense nonuniqueness]{Dense uniqueness and dense nonuniqueness\ of Fr\'echet means}
\hypersetup{pdftitle={Dense uniqueness and dense nonuniqueness of Frechet means}}
\author{Stephan F. Huckemann}
\author{Alexander Lytchak}
\date{\today} 

\begin{document}
\begin{abstract}
From previous results it follows that probability distributions on metric spaces featuring unique Fr\'echet means are dense among measures admitting means in both the quadratic Wasserstein metric and the total variation metric. Additionally, we show that the converse also holds on complete finite-dimensional noncontractible Alexandrov spaces with curvatures bounded from below, which encompass compact manifolds without boundaries and nonmanifold shape spaces: Probability distributions featuring nonunique Fr\'echet means are also dense in both the Wasserstein metric and the total variation metric.  Moreover, in either metric, no nonempty open set of probability measures admits a continuous selection of means. This sharpens a previous result on zero reach of the Dirac embedding.
Our argument combines cut-locus avoidance for atoms with a topological obstruction. 
For Riemannian manifolds, all of our conclusions hold for all exponents $1<p<\infty$, also.
\end{abstract}
\maketitle
\markboth{Dense uniqueness and dense nonuniqueness}{Dense uniqueness and dense nonuniqueness}

\section{Introduction}\label{sec:intro}

Let $(Q,d)$ be a metric space satisfying the \emph{Heine-Borel} property (every closed and bounded set is compact) and denote by $\Pcal(Q)$ the Borel probability measures on $Q$. For $\mu\in\Pcal(Q)$, its \emph{Fr\'echet function} is 
\begin{eqnarray*}
 Q\to [0,\infty], ~x \mapsto F_\mu(x)&:=&\int_Q d(x,z)^2\,d\mu(z),
\end{eqnarray*}
If $Q$ is not bounded we restrict ourselves to 
$$ \Pcal_2(Q) := \left\{\mu \in \Pcal(Q): F_\mu(x) < \infty\mbox{ for all } x\in Q\right\},$$
noting that $F_\mu(x) < \infty$ for $\mu \in \Pcal(Q)$ and at least one $x\in Q$ implies that $F_\mu(x) < \infty$ for all $x\in Q$. 
Further, we equip $\Pcal_2(Q)$ with the quadratic Wasserstein distance
\[
 W_2(\mu,\nu):=\sqrt{\inf_{\pi\in\Pi(\mu,\nu)}\int_{Q\times Q}d(z,w)^2\,d\pi(z,w)}\,,\quad \mu,\nu \in \Pcal_2(Q)\,,
\]
where $\Pi(\mu,\nu)$ is the set of couplings (e.g., \cite{PZ20} for this and further details below).
The \emph{Dirac embedding} $\delta: Q \hookrightarrow \Pcal_2(Q)$ identifies $x\in Q$ with the point measure $\delta_x$. By construction, it is isometric with respect to $W_2$. Since the only coupling of $\mu \in \Pcal_2(Q)$ with a Dirac measure is their product, 
$$W_2(\mu,\delta_x)^2=F_\mu(x)\,.$$
Moreover, for $\mu \in \Pcal_2(Q)$, 
\begin{eqnarray*}B(\mu)&:=&\argmin_{x\in Q}F_\mu(x)\end{eqnarray*}
is the set of its \emph{Fr\'echet means}, introduced by \cite{Fre48}. If $F_\mu(x_n) \to \inf_{x\in Q} F_\mu(x)$ for some $x_n \in Q$, $x_n$ has a cluster point $x^*\in Q$ with $F_\mu(x^*)=\inf_{x\in Q} F_\mu(x)$, due to the triangle inequality,
$$ d(y,x_n) = W_2(\delta_y,\delta_{x_n})\leq \sqrt{F_\mu(y)} +  \sqrt{F_\mu(x_n)}\,,$$
(for arbitrary $y\in Q$) and the Heine-Borel property. Hence $B(\mu)\neq \emptyset$ for all $\mu \in \Pcal_2(Q)$. 
Thus, $B(\mu)$ is the inverse image under $\delta$  of the nearest-point projection of $\mu$ to $\delta(Q)$.

An abundance of statistical applications (for a recent overview see \cite{MD21}) assumes uniqueness of $B(\mu)$, in order to provide a single data descriptor. An important question in this context is whether this descriptor remains unique under contamination, i.e., under perturbation within a neighborhood of the underlying distribution. 
We show that for a large family of statistically relevant sample spaces, for instance, compact manifolds without boundary (such as tori, spheres, real and complex projective spaces, e.g., \cite{MJ00}), noncontractible manifolds and \emph{shape spaces} \citep{DM16}, the answer is negative everywhere.

As the proof relies on noncontractibility and completeness, and as it utilizes previous results under curvature bounds from below (preventing forking geodesics, for this reason we need to exclude BHV spaces by \cite{BHV01}, say), we formulate our result in the following general setup
\begin{equation}\label{eq:setup}
 \left\{\begin{array}{l}
(Q,d)\mbox{\it ~is a noncontractible complete finite-dimensional}\\ \mbox{\it Alexandrov space of curvature bounded from below.}
 \end{array}\right. 
\end{equation}
This characterizes a finite-dimensional complete length space whose geodesic triangles satisfy a lower curvature comparison with the corresponding triangles in a constant-curvature model space; equivalently, its triangles are locally no thinner than the model triangles \citep[Def. 4.6.2]{BBI22}.  In particular, irrespective of noncontractibility, it then satisfies the Heine-Borel property \citep[Theorems 2.5.22 and 10.8.1]{BBI22}. 

On $\Pcal(Q)$ we also consider the \emph{total variation metric}
\[
d_{\mathrm{TV}}(\mu,\nu):=\sup_A|\mu(A)-\nu(A)|, \mbox{where $A$ ranges over all Borel sets of $Q$.}
\]
On a space of diameter $0<D<\infty$ (this is the case for Alexandrov spaces with positive lower curvature bound, cf. \citet[Theorem 10.4.1]{BBI22})  it induces a finer topology due to the well known \citep[Theorem 6.15]{V08}
\begin{eqnarray}\label{eq:W2_leq_TV} W_2(\mu,\nu)^2 &\leq& D^2 d_{\mathrm{TV}}(\mu,\nu)~~\mbox{ for all } \mu, \nu \in \Pcal(Q)\,.\end{eqnarray} 

Defining
\[
 \Ucal:=\{\mu\in \Pcal_2(Q):B(\mu)\text{ is a singleton}\},
 \qquad \Ncal:=\Pcal_2(Q)\setminus\Ucal\,,
\]
our first theorem asserts the coexistence of dense uniqueness and
dense instability.

\begin{maintheorem}\label{thm:main}
Under Setup (\ref{eq:setup}), the uniqueness locus $\Ucal$ and the nonuniqueness locus $\Ncal$ are both dense in $\Pcal_2(Q)$, both with respect to the quadratic Wasserstein metric and the total variation metric. Further, no nonempty open subset $\Vcal$ of $\Pcal_2(Q)$ admits a \emph{Fr\'echet mean map} $s:\Vcal\to Q$ with $s(\mu)\in B(\mu)$ that is continuous in either metric.
\end{maintheorem}

The proof of Theorem \ref{thm:main} in Section \ref{sec:genericity} gives a
perturbation result in which the size of the perturbation may be prescribed. To this end, for $\mu \in \Pcal_2(Q)$, $0\leq \varepsilon \leq 1$ and $y\in Q$ introduce the \emph{perturbed 
measure}
$$ \mu_{\varepsilon,y} := (1-\varepsilon)\mu+\varepsilon\delta_y\,.$$
\begin{maintheorem}\label{thm:contamination}
Under Setup (\ref{eq:setup}), there is a compact set $K\subseteq Q$ such that for every $\mu \in \Ucal$ and every sufficiently small $\varepsilon>0$, there exists $y\in K$ such that $\mu_{\varepsilon,y} \in \Ncal$.
\end{maintheorem}

Theorem~\ref{thm:contamination} asserts that uniqueness can be destroyed by a
single gross-error component of arbitrarily small positive weight.

A converse of Theorem \ref{thm:contamination} has been derived earlier on arbitrary metric spaces.

\begin{maintheorem}[{\citet[Lemma 1]{HEH26}}]\label{thm:pre-Ucal-dense} Let $(Q,d)$ be a metric space. For every $\mu\in \Ncal$, every $y\in B(\mu)$ and every $0<\varepsilon\leq 1$, we have $\mu_{\varepsilon,y} \in \Ucal$.
\end{maintheorem}

The proof of Theorem \ref{thm:contamination} in Section \ref{sec:genericity} is geometric:  Cut-locus avoidance of a measure implies that every one of its atoms is joined to each of its means by a unique minimizing
geodesic (\citet[Theorem~7.6]{LH25}, see also the preceding work of
\citet{LB14}).  If locally, all one-point perturbations had unique means,
moving the perturbation atom over a noncontractible compact set would allow means in a contractible neighborhood, and thus to construct a null-homotopic map that is homotopic to a non-null-homotopic map.  
Section~\ref{sec:extensions} gives examples and extends our results to other exponents $2\neq p \in (1,\infty)$ for Riemannian manifolds.

Topological obstructions to continuous mean maps were studied by
\cite{EGH62}:  A connected finite polyhedron $Q$
admitting a continuous symmetric map $m:Q^N\to Q$ with
$m(x,\ldots,x)=x$, for some $N\ge2$, is contractible.  If all uniform
$N$-point empirical measures on $Q$ had unique means, their means would define
such a map.  This is a global obstruction. Our
result is local: it excludes continuous
selections on every nonempty open subset of $\Pcal_2(Q)$.

\cite{CCS24} study the reach of the Dirac embedding into Wasserstein and Orlicz--Wasserstein spaces. In particular, for a closed Riemannian manifold $M$, they prove that the reach of a Dirac measure is zero, for every exponent greater than one \cite[Corollary 4.6]{CCS24}. For $p=2$, the reach of $\delta_x$ is the supremum of the radii $r$ for which the $W_2$-ball of radius $r$ about $\delta_x$ is contained in $\Ucal$. Their result therefore gives
\[
 \delta(M)\subset\overline{\Ncal}^{\,W_2},
\]
whereas Theorem~\ref{thm:main} gives
\[
 \overline{\Ncal}^{\,W_2}=\overline{\Ncal}^{\,\mathrm{TV}}=\Pcal(M).
\]
Thus, zero reach concerns nonuniqueness near the embedded manifold, while our conclusion concerns every probability measure. Their two-point construction starts at a Dirac measure; Theorem~\ref{thm:contamination} starts at an arbitrary measure with a unique mean and permits every sufficiently small perturbation weight.

There are various results on uniqueness of Fr\'echet means in the literature.  First,
uniqueness holds on Riemannian manifolds when the support is contained in a sufficiently small ball,
bounded in terms of curvature and injectivity radius; see
\cite{Afs11}.  Such support conditions are useful for concentrated
distributions, but confinement to a fixed proper subset has empty interior in
both the Wasserstein and total variation topologies: an arbitrarily small
amount of mass may be placed outside of it.

Second, in view of parametric statistics, distributions on symmetric spaces, featuring a density with respect to Riemannian volume that depends only on the distance to a central point that is nonincreasing and strictly decreasing on a set of positive volume have this point as their unique Fr\'echet mean. This follows at once from an argument by \cite{Le98} for \emph{planar shape spaces} (which are complex projective spaces).
In our contribution we do not extend this result to other distributions. In particular, we do not give an answer to questions of the following kind: Which distributional families or which parameter choices for certain parametric families guarantee unique means? Such questions are left for future research.

Third, \citet[Theorem~2.1]{AM14} prove that the uniform empirical measure associated
with almost every configuration in $M^N$, with respect to product Riemannian
volume, has a unique Fr\'echet mean.  Thus, almost-sure uniqueness at fixed
sample size $N\in \mathbb N$ and topological density of nonuniqueness in the full distribution
space occur simultaneously. In contrast, Corollary~\ref{cor:finite} shows that finitely supported measures with nonunique means are dense in general noncontractible complete finite-dimensional Alexandrov spaces $Q$ with curvature bounded from below.  

Statistical consequences of nonuniqueness have been
quantified by \cite{HEH26}, who derive uniform
estimation lower bounds near distributions with several means. \cite{Elt20} provides statistical testing for uniqueness of means.

\section{Continuity, atoms, and the topological obstruction}\label{sec:atoms}

For all of the following, $(Q,d)$ is a complete finite-dimensional Alexandrov space of curvature bounded from below. 
Recall from above that it then satisfies the Heine-Borel property.

\begin{lemma}\label{lem:unique-means}
 The unique mean map 
 $$ b:\Ucal \to Q,\mu\mapsto \bar \mu\mbox{ where }\{\bar \mu\} = B(\mu)$$
 is continuous in the $W_2$ metric. 
\end{lemma}

\begin{proof}
For every $\nu \in \Pcal_2(Q)$, $\bar \nu \in B(\nu)$ and $y\in Q$, we have
$$ d(y,\bar \nu) = W_2(\delta_y,\delta_{\bar \nu}) \leq \sqrt{F_\nu(y)} + \sqrt{F_\nu(\bar \nu)} \leq 2\sqrt{F_\nu(y)}\,,$$
due to the triangle inequality. Hence,  for $\Pcal_2(Q) \ni \mu_n \to \mu \in \Pcal_2(Q)$ with $B(\mu_n) = \{\bar \mu_n\}, B(\mu) = \{\bar \mu\}$, we have that all $\bar\mu_n$ lie in a closed and bounded subset. Due to the Heine-Borel property, $\bar\mu_n$ has then a cluster point $\bar \mu^*$ with $F_\mu(\bar \mu^*) = F_\mu(\bar\mu)$. The latter follows from the triangle inequality
$$ \left|\sqrt{F_\nu(x)} - \sqrt{F_\tau(x)}\right| \leq W_2(\nu,\tau)\,,$$
for all $\nu,\tau \in \Pcal_2(Q)$, $x\in Q$, and applying
$$ \min_{x\in Q} \sqrt{F_\nu(x)} \leq  \min_{x\in Q} \sqrt{F_\tau(x)} + W_2(\nu,\tau)$$
for suitable $\nu,\tau\in \Pcal_2(Q)$, to
obtain
\begin{eqnarray*}
 F_{\mu_n}(\bar \mu_n) ~=~ \min_{x\in Q} F_{\mu_n}(x) &\to& \min_{x\in Q} F_{\mu}(x) ~=~ F_{\mu}(\bar \mu)\,.\end{eqnarray*}
Due to uniqueness, $\bar \mu^* = \bar \mu$, and hence $\bar\mu_n \to \bar \mu$, yielding continuity of the unique mean map, $b:\Ucal \to Q$ in $W_2$.
\end{proof}

The following is a generalization of \citet[Theorem~1]{LB14}.

\begin{lemma}[{\citet[Theorem~7.6]{LH25}}]\label{lem:atom}
If $x$ is a minimizer of $F_\mu$ for $\mu \in \Pcal_2(Q)$ with an atom $\mu(\{y\})>0$ at $y\in Q$, then $x$ and $y$ are joined by a unique minimizing geodesic.
\end{lemma}

We now formulate and prove the central topological observation.

\begin{lemma}\label{lem:homotopy}
Let $K$ be a compact metric space and let $f_0,f_1:K\to Q$ be continuous. If $f_0(z)$ and $f_1(z)$ are joined by a unique minimizing geodesic for every $z\in K$, then $f_0$ and $f_1$ are homotopic.
\end{lemma}

\begin{proof}
For every $z \in K$ let $\gamma_z:[0,1]\to Q$ be the unique constant speed geodesic joining $f_0(z)$ and $f_1(z)$. 
Since $K$ is compact there is $L< \infty$ such that
$$ \max_{z \in  K} d\left(f_0(z),f_1(z)\right) \leq L\,.$$
Hence for all $z\in K$,
$$ d(\gamma_z(s),\gamma_z(t)) = |s-t|  d\left(f_0(z),f_1(z)\right) \leq L|s-t|\,,$$
yielding that the $\gamma_z$, $z\in K$, are equicontinuous. Since $K$ is compact, there is $R>0$ and $x\in Q$ such that $f_0(K)\subset\{y\in Q: d(x,y)\leq R\}$, yielding that
$$ \gamma_z(t) \in \{y\in Q: d(x,y)\leq R+L\}\mbox{ for all }z\in K,t\in [0,1]\,.$$
Due to the Heine-Borel property, the above r.h.s. is a compact space, and since the $\gamma_z$ ($z\in K$) are equicontinuous, by the Arzel\`a-Ascoli theorem, whenever $K\ni z_n \to z \in K$, there is a subsequence $j_k$ such that $\gamma_{z_{j_k}}$ converges to some continuous curve $\eta: [0,1]\to Q$ joining  $f_0(z)=\eta(0)$ with $f_1(z)=\eta(1)$. As a limit of constant speed minimizing geodesics,  this curve is also a constant speed minimizing geodesic, and since there is a unique constant speed minimizing geodesic $\gamma_z: [0,1]\to Q$ joining $f_0(z)=\gamma_z(0)$ with $f_1(z)=\gamma_z(1)$, $\eta = \gamma_z$. If there were another subsequence not converging to $\gamma_z$, there would be, again by the Arzel\`a-Ascoli theorem, a subsequence of this subsequence that would converge to some constant speed minimizing geodesic, which, by uniqueness, is again $\gamma_z$. Thus $K \to Q, z\mapsto \gamma_z(t)$ is uniformly continuous for every $t\in [0,1]$.

In consequence, we have shown that  
$$ H(z,t) := \gamma_z(t)\mbox{ for }z\in K,t\in [0,1]$$
is a homotopy from $f_0(z)=H(z,0)$ to $H(z,1) =f_1(z)$, since $H(z,\cdot) = \gamma_z$ is $L$-Lipschitz continuous for all $z\in K$ (see above), and so is $H(\cdot,t)$ for all $t\in [0,1]$, as we just showed.
\end{proof}

\section{Perturbation and genericity}\label{sec:genericity}

\begin{proof}[Proof of Theorem~\ref{thm:contamination}]
\emph{Uniform localization.}
Every point in $Q$  has a contractible neighborhood \citep[Theorem 3.1]{MY19}. As $Q$ is non-contractible by hypothesis, there exists $k\in \mathbb N$ and a non-null-homotopic map $u:\mathbb S^k\to Q$ \citep[Theorem 9.2]{Du58}.
Let $u$ be such a map and set $K:=\mathbb{S}^k$.

Suppose now that $B(\mu)=\{\bar \mu\}$. Choose a relatively compact contractible open neighborhood $C$ of $\bar \mu$. Since $d(x_n,\bar \mu)\to \infty$ implies that $F_\mu(x_n)\to \infty$ for any sequence $x_n \in Q$, the Heine-Borel property and uniqueness give
\[
 a:=\inf_{x\notin C}\bigl(F_\mu(x)-F_\mu(\bar \mu)\bigr)>0.
\]
With $R:=\max_{y\in K}d(\bar \mu,y)$, $0<\varepsilon<a/(a+R^2)$, $y\in K$, consider $\mu_{\varepsilon,y}=(1-\varepsilon)\mu+\varepsilon\delta_y$. Since
\[
 F_{\mu_{\varepsilon,y}}(x)-F_{\mu_{\varepsilon,y}}(\bar \mu)
 \ge(1-\varepsilon)a-\varepsilon R^2>0\,,
\]
for all $x\notin C$, we have $B(\mu_{\varepsilon,y})\subset C$ for every $y\in K$.

\emph{The hypothetical mean map.}
If all these mean sets were singletons $\bar \mu_{\varepsilon,y}$, they would define a ``unique-mean'' map
$f:K\to C, y \mapsto \mu_{\varepsilon,y} \mapsto \bar \mu_{\varepsilon,y}$ which is continuous, since the first mapping is continuous due to 
\[
 W_2(\mu_{\varepsilon,y},\mu_{\varepsilon,z})
 \le\sqrt\varepsilon\,d(y,z)\,,
\]
and the second mapping is a restriction of $b$ which is continuous, due to Lemma  \ref{lem:unique-means}.

\emph{The homotopy contradiction.}
Since $\mu_{\varepsilon,y}(\{y\})\ge\varepsilon$,
Lemma~\ref{lem:atom} gives a unique minimizing geodesic from $y$ to $f(y)$.
With the above constructed $K$, Lemma~\ref{lem:homotopy} makes $f_1 = u$ homotopic to $f_2 = f\circ u$, leading to the contradiction that a non-null-homotopic map is homotopic to a map into a contractible set $C$.
\end{proof}

Recall that a \emph{closed manifold} is a compact manifold without boundary.
Since every closed manifold that is more than one point is not contractible \citep[Chapter 2, \S~4, Exercise 6]{GP74}, we have at once:

\begin{corollary}\label{cor:thmB-4-compact-manifolds}
 Let $M$ be a closed Riemannian manifold of dimension at least one and $\mu \in \Pcal(M)$ with $\{\bar \mu\} = B(\mu)$. Then for every sufficiently small $\varepsilon > 0$ there is $y\in M$ such that $B(\mu_{\varepsilon,y})$ contains at least two points.
\end{corollary}

\begin{proof}[Proof of Theorem~\ref{thm:main}]
For all $\mu, \nu \in \Pcal_2(Q)$, $y\in Q$ and $0<\varepsilon<1$  we have
\begin{eqnarray*}
 d_{\mathrm{TV}}(\mu_{\varepsilon,y},\mu)\le\varepsilon\,,&& W^2_2(\mu_{\varepsilon,y},\mu)\le \varepsilon\, F_\mu(y)\,,
\end{eqnarray*}
where the second estimate uses the coupling that leaves $(1-\varepsilon)\mu$ fixed and transports the remaining mass to $y$. In conjunction with Theorem \ref{thm:contamination} this yields that $\Ncal$ is dense in both metrics, and with Theorem \ref{thm:pre-Ucal-dense} that $\Ucal$ is dense in both metrics.

Moreover, if $x,y\in B(\mu)$ are distinct, the measures $\mu_{\varepsilon,x}$ and $\mu_{\varepsilon,y}$ approach $\mu$ in either metric and have unique means $x$ and $y$. Thus, a selection on any neighborhood of $\mu$ cannot be continuous there in either metric.
\end{proof}

Here is a contrast to the Arnaudon-Miclo theorem \citep{AM14} that sample Fr\'echet means of measures absolutely continuous with respect to Riemannian volume are a.s. unique.

\begin{corollary}\label{cor:finite}
Let $Q$ be a noncontractible complete finite-dimensional Alexandrov space. Then finitely supported measures with nonunique means are dense in $\Pcal_2(Q)$ with respect to $W_2$.
\end{corollary}

\begin{proof}
By \citet[Theorem 6.18]{V08} the subspace of finite discrete measures is dense in $\Pcal_2(Q)$. Applying Theorem~\ref{thm:contamination} to those with a unique mean results in adding one atom, which preserves finite support.
\end{proof}

\section{Interpretation, examples, and extensions}\label{sec:extensions}

\subsection{Spheres and intervals}
The following is the spherical case of the two-point construction in \cite{CCS24}. On the unit sphere $\mathbb{S}^n$, $n\ge1$, $y\in \mathbb{S}^n$, let $\nu_\varepsilon=(1-\varepsilon)\delta_y+\varepsilon\delta_{-y}$, where $0<\varepsilon<1$. For $r=d(y,x)$,
\[
 F_{\nu_\varepsilon}(x)
 =(1-\varepsilon)r^2+\varepsilon(\pi-r)^2
 =(r-\varepsilon\pi)^2+\varepsilon(1-\varepsilon)\pi^2.
\]
Thus
\[
 B(\nu_\varepsilon)=\{x:d(y,x)=\varepsilon\pi\},
 \qquad W_2(\nu_\varepsilon,\delta_y)=\pi\sqrt\varepsilon.
\]
The mean set consists of two points for $n=1$ and is a sphere of codimension one for $n\ge2$.

This example illustrates population instability, while there is no claim that the resulting mean
sets have a uniformly positive diameter, nor that nonunique empirical means
occur with positive probability for fixed sample size.

Our assertions fail for compact manifolds with boundary in general. On $Q=[0,1]$, setting $m_\mu=\int z\,d\mu(z)$ gives
\[
 F_\mu(x)=(x-m_\mu)^2+\int(z-m_\mu)^2\,d\mu(z).
\]
Hence every measure has the unique mean $m_\mu$, and projection from $\Pcal([0,1])$ to its Dirac measures is single-valued everywhere. 

Note that $[0,\pi]$ can be viewed as the quotient of $\mathbb S^2$ modulo $\mathbb S^1$ (rotations about a fixed axis), a ``toy'' shape space. Beginning with two-dimensional configurations, however, the ``real'' Kendall's shape spaces are noncontractible, see below. 

More generally, in a complete CAT(0) space every probability measure with finite second moment has a unique Fr\'echet mean, and the Fr\'echet mean map (see Theorem \ref{thm:main}) is $1$-Lipschitz with respect to $W_2$; see \citet[Proposition~4.3 and Theorem~6.3]{Stu03}. 

\subsection{Kendall's shape spaces}

    Consider \emph{configuration matrices} $X = (x_1,\ldots,x_k)$ of $k$ \emph{landmarks} $x_j \in \R^m$ ($j=1,\ldots,k$) in $m$ dimensions not all of them coinciding, modulo translation, scaling and rotation, yielding Kendall's shape space  {\citep{KBCL99}}
    $$ \Sigma_m^k := \left\{[X] : X\in \R^{m\times k} \setminus \{(x,\ldots,x) \in \R^{m\times k}: x \in  \R^m\}\right\}\,.$$
    Here 
    $$ [X] := \left\{Y \in \R^{m\times k}: g X P_k = \lambda Y P_k \mbox{ for some }\lambda > 0\mbox{ and }g\in SO(m)\right\}\,,$$
    with the centering projection $P_k := (I_k - \frac{1}{k}1_k1^T_k)$ where $I_k$ is the $k\times k$ unit matrix and $1_k\in\R^k$ has all coordinates equal one. 
    In case of $m=2<k$ these spaces can be given the manifold structure of complex projective spaces of complex dimension $k-2$ \citep[Section 6.4]{KBCL99}. This is a compact manifold, so Theorems \ref{thm:main}, \ref{thm:contamination} hold, as well as, Corollary \ref{cor:finite}.
    
    More generally, with the Euclidean norm $\|\cdot\|$ in matrix space, 
    $$Z_{\Sigma_m^k}([X],[Y]) := \mathop{\min}_{g\in  SO(m)} \left\|\frac{gXP_k}{\|XP_k\|} - \frac{YP_k}{\|YP_k\|}\right\|\quad \mbox{for }[X],[Y] \in \Sigma_m^k$$
    defines the \emph{Ziezold} metric on $\Sigma_m^k$, yielding the intrinsic metric $d_{\Sigma_m^k}:=2\arcsin(Z_{\Sigma_m^k}/2)$ \citep{H_means_means}.
    The following is well known, some of which is folklore.
    
    \begin{theorem}\label{thm:shape-spaces}
     In case of $3\le m< k-1$, $(\Sigma_m^k,d_{\Sigma_m^k})$ cannot be given a topological manifold structure. Moreover, for $2\le m\leq k-1$, $(\Sigma_m^k,d_{\Sigma_m^k})$ has the structure of a complete finite-dimensional noncontractible Alexandrov space with curvatures bounded from below by $1$.
    \end{theorem}

    \begin{proof}
     The nonmanifold assertion is from \citet[Theorem 4.6]{KBCL99}. Since $\Sigma_m^k$ can be viewed as a quotient of $SO(m)$ acting on a unit sphere in matrix space, it is a complete finite-dimensional Alexandrov space with curvatures bounded from below by $1$ \citep[Proposition 10.2.4]{BBI22}. It is noncontractible because it has nontrivial homology groups \citep[Theorem 4.5]{KBCL99}. 
    \end{proof}

    Note that $\Sigma_m^{m+1}$ is homeomorphic to a sphere \citep{Le91}.

\subsection{Other exponents}
For $1<p<\infty$, $\mu \in \Pcal(Q)$ and $x\in Q$, set
$$
 F_{\mu,p}(x):=\int_Qd(x,z)^p\,d\mu(z)\,,$$
 and 
$$\Pcal_p(Q):=\{\mu \in \Pcal(Q): F_{\mu,p}(x) <\infty \mbox{ for all }x\in Q\}\,.
$$
Then, the 
$$B_p(\mu):=\argmin_{x\in Q}F_{\mu,p}(x)$$
are the nearest-point sets for the Dirac embedding $\delta:Q\hookrightarrow \Pcal_p(Q)$,
using the metric
$$  W_p(\mu,\nu):=\left(\inf_{\pi\in\Pi(\mu,\nu)}\int_{Q\times Q}d(z,w)^p\,d\pi(z,w)\right)^{\frac{1}{p}}\,,\quad \mu,\nu \in \Pcal_p(Q)\,.$$

\begin{theorem}\label{thm:otherp}
On a complete noncontractible  Riemannian manifold $Q$ with intrinsic distance $d$, Theorems~\ref{thm:main} and~\ref{thm:contamination} and Corollaries~\ref{cor:thmB-4-compact-manifolds} and \ref{cor:finite} hold with 
$$ \Ucal:=\{\mu\in \Pcal_p(Q):B_p(\mu)\text{ is a singleton}\},
 \qquad \Ncal:=\Pcal_p(Q)\setminus\Ucal\,,$$
and $W_2,B$ replaced by $W_p,B_p$, respectively, for every $1<p<\infty$.
\end{theorem}

\begin{proof}
Lemma~\ref{lem:atom} holds for every $p>1$ by \citet[Section~7, ``Exponents different from two'']{LH25}. Then all arguments carry through, using the estimates
\begin{align*}
 d(y,\bar \nu) &= W_p(\delta_y,\delta_{\bar \nu}) \leq (F_{\nu,p}(y))^{\frac{1}{p}} + (F_{\nu,p}(\bar \nu))^{\frac{1}{p}},\\
  \left|(F_{\nu,p}(x))^{\frac{1}{p}} - (F_{\tau,p}(x))^{\frac{1}{p}}\right| &\leq W_p(\nu,\tau)\,,\\
 W^p_p(\mu_{\varepsilon,y},\mu)&\le \varepsilon F_{\mu,p}(y),\\
 W_p(\mu_{\varepsilon,y},\mu_{\varepsilon,z})&\le\varepsilon^{1/p}d(y,z).
\end{align*}
\end{proof}
Note that the total variation bounds do not depend on $p$.

For $p=1$ Lemma~\ref{lem:atom} fails. On the unit sphere $M=\mathbb{S}^n$, $n \in \mathbb N$, for fixed $y\in \mathbb{S}^n$ the measure $\eta_\varepsilon=(1-\varepsilon)\delta_y+\varepsilon\delta_{-y}$ has the unique median $y$ for $0<\varepsilon<1/2$, since
\[
 F_{\eta_\varepsilon,1}(x)-F_{\eta_\varepsilon,1}(y)
 \ge(1-2\varepsilon)d(x,y).
\]
Nevertheless, its antipodal atom $-y$ is joined to $y$ by more than one minimizing geodesic.

\section*{Acknowledgements}

Both authors acknowledge a fruitful stay at the Hausdorff Research Institute for Mathematics in Bonn supported by the Deutsche Forschungsgemeinschaft (DFG, German Research Foundation) under Germany's Excellence Strategy -- EXC-2047/1 -- 390685813.

This paper arose in close collaboration with ChatGPT, an AI assistant developed by OpenAI. The collaboration involved exploring proof strategies, refining arguments, locating references, and drafting. 
Revision and writing of final text have been done by the authors.

\end{document}